\documentclass[letterpaper,12pt]{article}

\usepackage[margin=0.5in]{geometry}

\RequirePackage{etex}

\usepackage{fullpage}
\usepackage{enumerate}
\usepackage{authblk}
\usepackage[colorinlistoftodos]{todonotes}
\usepackage{verbatim}
\usepackage{section, amsthm, textcase, setspace, amssymb, lineno, amsmath, amssymb, amsfonts, latexsym, fancyhdr, longtable, ulem}
\usepackage{epsfig, graphicx, pstricks,pst-grad,pst-text,tikz, tkz-berge,colortbl}
\usepackage{epsf}
\usepackage{graphicx, color}
\usepackage{float}
\usepackage[rflt]{floatflt}
\usepackage{amsfonts}
\usepackage{latexsym}

\definecolor{darkgreen}{rgb}{0, 0.5, 0}

\theoremstyle{plain}
\newtheorem{theorem}{Theorem}
\numberwithin{theorem}{section}

\newtheorem{corollary}[theorem]{Corollary}
\newtheorem{proposition}[theorem]{Proposition}

\theoremstyle{definition}
\newtheorem{definition}[theorem]{Definition}

\theoremstyle{remark}

\renewcommand{\P}{\mathcal{P}}

\title{How to Build a Graph in $n$ Days:  Some Variations of Graph Assembly}

\author{Aria Dougherty}
\author{Nick Mayers}
\author{Robert Short}

\affil{Department of Mathematics, Lehigh University, Bethlehem, PA, 18015}

\begin{document}
\maketitle

\noindent
\begin{abstract}
\noindent
In a recent article by M. Bona and A. Vince, the authors introduced the concept of an assembly tree for a graph. Assembly trees act as record keeping devices for the construction of a given graph from its vertices. In this paper we extend the work initiated by M. Bona and A. Vince as well as define and examine a more generalized definition of an assembly tree.
\end{abstract}


\bigskip
\noindent

\tableofcontents

\section{Introduction}

As the title suggests, the subject matter of this paper is graph-theoretic, but the motivation for our objects of interest, in fact, comes from the study of viruses. A detailed discussion on viruses and their connection to graph theory can be found in a paper by Bona et al \cite{BSV11}.

Encouraged by the behavior of viruses, Bona et al \cite{BSV11} introduced the notion of an assembly tree of a graph. An assembly tree for a graph, $G$, can be thought of as a record of how $G$ can be ``assembled" starting from the set of its vertices; but what rules should dictate how a graph is assembled from its vertices?

Bona and Vince \cite{BV13} answer this question by defining ``gluing rules". They focus on one such rule, requiring a graph to be assembled in such a way that two groups of vertices are combined as long as they are connected by an edge. Using this restriction, the authors determine explicit formulas or generating functions for the number of assembly trees for many classical families of graphs. In concluding their results, Bona and Vince offer alternative gluing rules.

In this paper, we begin by analyzing one such alternative gluing rule which they call the \textit{connected gluing rule}. Given a graph $G$, using the connected gluing rule, at each stage of construction each specified grouping of vertices of $G$ must form a connected induced subgraph of $G$. Investigating various classical families of graphs, we determine formulas and recursive relations for the number of such assembly trees. For some of these families our formulas are related to classic integer sequences in combinatorics, while for the others these sequences are new to the literature. 

In the case of star graphs, the resulting integer sequence forms what is known as the sequence of Fubini numbers\footnote{The Fubini numbers (A000670), so named by Louis Comtet, count the number of different ways to rearrange the orderings of sums or integrals in Fubini's Theorem.}. Curiously, this sequence was studied by Cayley  (\cite{C2},1859) while enumerating a certain family of trees. In this particular case, Cayley was studying a family of trees which could be realized as a generalization of assembly trees on paths where a notion of time or order is also tracked. In Section \ref{timdepprelim} we define such a generalized assembly tree and enumerate them for various families of graphs with both the edge and connected gluing rules. As before, we encounter a mix of both old and new integer sequences. 

This paper is organized as follows. In Section~\ref{prelim} the preliminaries of graph theory needed for this paper are given, along with a formal treatment of assembly trees and gluing rules. In Section~\ref{enumer} we prove our main results on the enumeration of assembly trees for the classical families of stars, paths, cycles, and complete graphs with the connected gluing rule. Following this, in Section~\ref{timdepprelim} we introduce the notion of a time-dependent assembly tree and enumerate such trees for the same four families of graphs with the addition of both the edge and connected gluing rule. We conclude in Section~\ref{seccon} by discussing further research directions.

\section{Preliminaries}\label{prelim}

Throughout this paper, all graphs $G$ are assumed to be simple, i.e., will have no loops or multiple edges. Let $G=(V,E)$ be the graph with vertex set $V$ and edge set $E$. In the definition of an assembly tree $T$ for a graph $G$, each vertex of $T$ is labeled by a subset of $V$.  No distinction will be made between a vertex and its label.  For  a vertex $U$ in a rooted tree, $c(U)$ denotes the children of $U$.

\begin{definition} 
Let $G$ be a connected graph on $n$ vertices.  An assembly tree for $G$ 
is a rooted tree, each vertex of which is labeled by a subset $U\subseteq V$ such that
\begin{enumerate}
\item the label of the root is $[n]=\{1, 2, \dots, n\}$,
\item each internal vertex $U$ has at least two children
and (the label of) $U=\bigcup c(U)$,
\item there are $n$ leaves which are labeled $1, 2, 3, \dots, n$, respectively.
\end{enumerate}
\end{definition}

The following Figure \ref{assemtree} illustrates an assembly tree  on a graph with seven vertices. Here, there are three internal nodes:
$U_1 = \{ 1, 3, 5, 7 \}$, $U_2 = \{4, 2, 6 \}$, and $U_3 = \{ 2, 6 \}$.

\begin{figure}[H]
$$\begin{tikzpicture}
	\node (1) at (0, 3) [circle, draw = black, fill=black, inner sep = 0.5mm, label=above:{\tiny\{1,2,3,4,5,6,7\}}] {};
	\node (2) at (-.75,2.3) [circle, draw = black, fill=black,  inner sep = 0.5mm, label=left:{\tiny\{1,3,5,7\}}] {};
	\node (3) at (1, 2.3) [circle, draw = black, fill=black, inner sep = 0.5mm, label=right:{\tiny\{4,2,6\}}] {};
    \node (4) at (1.25, 1.5) [circle, draw = black, fill=black, inner sep = 0.5mm, label=right:{\tiny\{2,6\}}] {};
    \node (5) at (-1.5, 1) [circle, draw = black, fill=black, inner sep = 0.5mm, label=below:{1}] {};
    \node (6) at (-1, 1) [circle, draw = black, fill=black, inner sep = 0.5mm, label=below:{3}] {};
    \node (7) at (-.5, 1) [circle, draw = black, fill=black, inner sep = 0.5mm, label=below:{5}] {};
    \node (8) at (0, 1) [circle, draw = black, fill=black, inner sep = 0.5mm, label=below:{7}] {};
    \node (9) at (.5, 1) [circle, draw = black, fill=black, inner sep = 0.5mm, label=below:{4}] {};
    \node (10) at (1, 1) [circle, draw = black, fill=black, inner sep = 0.5mm, label=below:{2}] {};
    \node (11) at (1.5, 1) [circle, draw = black, fill=black, inner sep = 0.5mm, label=below:{6}] {};
	\draw (2)--(1)--(3);
    \draw (9)--(3)--(4);
    \draw (5)--(2)--(8);
    \draw (6)--(2)--(7);
    \draw (10)--(4)--(11);
\end{tikzpicture}$$
\caption{An assembly tree on a graph of seven vertices}\label{assemtree}
\end{figure}
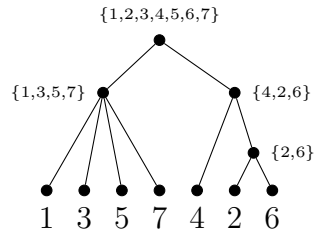

Evidently, all graphs with the same number of vertices have the same set of assembly trees. Thus, it seems beneficial to add additional constraints which force properties of the graph to have an impact on the structure of the corresponding assembly tree. For this reason Bona and Vince \cite{BV13} defined additional ``gluing rules'' which dictate the pairs of subsets of vertices that can be combined. 

Bona and Vince focused on what they called the ``edge gluing rule''. Given a graph $G=(V,E)$, an assembly tree $T$ for $G$ satisfies the \textit{edge gluing rule} if each internal vertex $v\in T$ has exactly two children $U_1$ and $U_2$ such that there is an edge $v_1v_2\in E$, called the the \textit{gluing edge}, such that $v_1\in U_1$ and $v_2\in U_2$. In Figure~\ref{K3EGR}, the assembly trees of the complete graph on 3 vertices, $K_3$, are given assuming the addition of the edge gluing rule. 

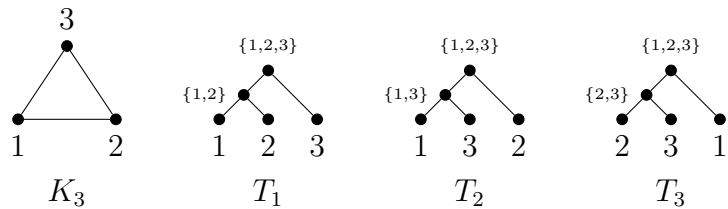
\begin{figure}[H]
$$\begin{tikzpicture}[scale = 0.65]
	\node (1) at (0, 0) [circle, draw = black, fill=black, inner sep = 0.5mm, label=below:{1}] {};
	\node (2) at (2,0) [circle, draw = black, fill=black,  inner sep = 0.5mm, label=below:{2}] {};
	\node (3) at (1, 1.5) [circle, draw = black, fill=black, inner sep = 0.5mm, label=above:{3}] {};
    \node (4) at (1, -1.5) {$K_3$};
	\draw (1)--(2);
    \draw (2)--(3);
    \draw (3)--(1);
\end{tikzpicture}
\hspace{0.5cm}
\begin{tikzpicture}[scale = 0.65]
	\node (1) at (0, 0) [circle, draw = black, fill=black, inner sep = 0.5mm, label=below:{1}] {};
	\node (2) at (1,0) [circle, draw = black, fill=black,  inner sep = 0.5mm, label=below:{2}] {};
	\node (3) at (2, 0) [circle, draw = black, fill=black, inner sep = 0.5mm, label=below:{3}] {};
    \node (4) at (0.5, 0.5) [circle, draw = black, fill=black, inner sep = 0.5mm, label=left:{\tiny \{1,2\}}] {};
    \node (5) at (1, 1) [circle, draw = black, fill=black, inner sep = 0.5mm, label=above:{\tiny \{1,2,3\}}] {};
    \node (6) at (1, -1.5) {$T_1$};
	\draw (1)--(4);
    \draw (2)--(4);
    \draw (4)--(5);
    \draw (3)--(5);
\end{tikzpicture}
\hspace{0.5cm}
\begin{tikzpicture}[scale = 0.65]
	\node (1) at (0, 0) [circle, draw = black, fill=black, inner sep = 0.5mm, label=below:{1}] {};
	\node (2) at (1,0) [circle, draw = black, fill=black,  inner sep = 0.5mm, label=below:{3}] {};
	\node (3) at (2, 0) [circle, draw = black, fill=black, inner sep = 0.5mm, label=below:{2}] {};
    \node (4) at (0.5, 0.5) [circle, draw = black, fill=black, inner sep = 0.5mm, label=left:{\tiny \{1,3\}}] {};
    \node (5) at (1, 1) [circle, draw = black, fill=black, inner sep = 0.5mm, label=above:{\tiny \{1,2,3\}}] {};
    \node (6) at (1, -1.5) {$T_2$};
	\draw (1)--(4);
    \draw (2)--(4);
    \draw (4)--(5);
    \draw (3)--(5);
\end{tikzpicture}
\hspace{0.5cm}
\begin{tikzpicture}[scale = 0.65]
	\node (1) at (0, 0) [circle, draw = black, fill=black, inner sep = 0.5mm, label=below:{2}] {};
	\node (2) at (1,0) [circle, draw = black, fill=black,  inner sep = 0.5mm, label=below:{3}] {};
	\node (3) at (2, 0) [circle, draw = black, fill=black, inner sep = 0.5mm, label=below:{1}] {};
    \node (4) at (0.5, 0.5) [circle, draw = black, fill=black, inner sep = 0.5mm, label=left:{\tiny \{2,3\}}] {};
    \node (5) at (1, 1) [circle, draw = black, fill=black, inner sep = 0.5mm, label=above:{\tiny \{1,2,3\}}] {};
    \node (6) at (1, -1.5) {$T_3$};
	\draw (1)--(4);
    \draw (2)--(4);
    \draw (4)--(5);
    \draw (3)--(5);
\end{tikzpicture}$$
\caption{Edge Gluing Assembly Trees for $K_3$}\label{K3EGR}
\end{figure}

While focusing their investigation on assembly trees satisfying the edge gluing rule, in their conclusion Bona and Vince offer some reasonable alternative options for gluing rules. Among these alternative rules is the ``connected gluing rule". If $G=(V, E)$ is a graph, then $V'\subseteq V$ defines an \textit{induced subgraph} consisting of $V'$ and the set of all edges defined by elements of $V'$. We say that an assembly tree satisfies the \textit{connected gluing rule} if for each internal node, the graph induced by the vertices in the label is connected. In Figure~\ref{K3CGR}, the assembly trees of $K_3$ are given, this time assuming the addition of the edge gluing rule.

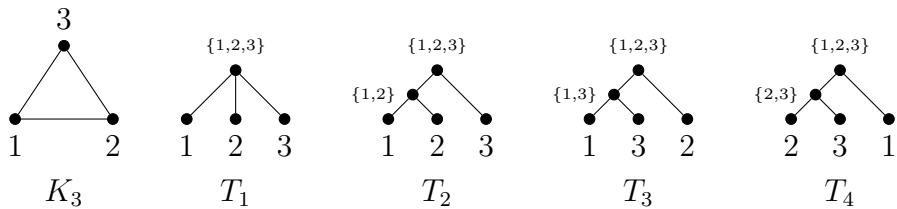
\begin{figure}[H]
$$\begin{tikzpicture}[scale = 0.65]
	\node (1) at (0, 0) [circle, draw = black, fill=black, inner sep = 0.5mm, label=below:{1}] {};
	\node (2) at (2,0) [circle, draw = black, fill=black,  inner sep = 0.5mm, label=below:{2}] {};
	\node (3) at (1, 1.5) [circle, draw = black, fill=black, inner sep = 0.5mm, label=above:{3}] {};
    \node (4) at (1, -1.5) {$K_3$};
	\draw (1)--(2);
    \draw (2)--(3);
    \draw (3)--(1);
\end{tikzpicture}
\hspace{0.5cm}
\begin{tikzpicture}[scale = 0.65]
	\node (1) at (0, 0) [circle, draw = black, fill=black, inner sep = 0.5mm, label=below:{1}] {};
	\node (2) at (1,0) [circle, draw = black, fill=black,  inner sep = 0.5mm, label=below:{2}] {};
	\node (3) at (2, 0) [circle, draw = black, fill=black, inner sep = 0.5mm, label=below:{3}] {};
    \node (4) at (1, 1) [circle, draw = black, fill=black, inner sep = 0.5mm, label=above:{\tiny \{1,2,3\}}] {};
    \node (5) at (1, -1.5) {$T_1$};
	\draw (1)--(4);
    \draw (2)--(4);
    \draw (3)--(4);
\end{tikzpicture}
\hspace{0.5cm}
\begin{tikzpicture}[scale = 0.65]
	\node (1) at (0, 0) [circle, draw = black, fill=black, inner sep = 0.5mm, label=below:{1}] {};
	\node (2) at (1,0) [circle, draw = black, fill=black,  inner sep = 0.5mm, label=below:{2}] {};
	\node (3) at (2, 0) [circle, draw = black, fill=black, inner sep = 0.5mm, label=below:{3}] {};
    \node (4) at (0.5, 0.5) [circle, draw = black, fill=black, inner sep = 0.5mm, label=left:{\tiny \{1,2\}}] {};
    \node (5) at (1, 1) [circle, draw = black, fill=black, inner sep = 0.5mm, label=above:{\tiny \{1,2,3\}}] {};
    \node (6) at (1, -1.5) {$T_2$};
	\draw (1)--(4);
    \draw (2)--(4);
    \draw (4)--(5);
    \draw (3)--(5);
\end{tikzpicture}
\hspace{0.5cm}
\begin{tikzpicture}[scale = 0.65]
	\node (1) at (0, 0) [circle, draw = black, fill=black, inner sep = 0.5mm, label=below:{1}] {};
	\node (2) at (1,0) [circle, draw = black, fill=black,  inner sep = 0.5mm, label=below:{3}] {};
	\node (3) at (2, 0) [circle, draw = black, fill=black, inner sep = 0.5mm, label=below:{2}] {};
    \node (4) at (0.5, 0.5) [circle, draw = black, fill=black, inner sep = 0.5mm, label=left:{\tiny \{1,3\}}] {};
    \node (5) at (1, 1) [circle, draw = black, fill=black, inner sep = 0.5mm, label=above:{\tiny \{1,2,3\}}] {};
    \node (6) at (1, -1.5) {$T_3$};
	\draw (1)--(4);
    \draw (2)--(4);
    \draw (4)--(5);
    \draw (3)--(5);
\end{tikzpicture}
\hspace{0.5cm}
\begin{tikzpicture}[scale = 0.65]
	\node (1) at (0, 0) [circle, draw = black, fill=black, inner sep = 0.5mm, label=below:{2}] {};
	\node (2) at (1,0) [circle, draw = black, fill=black,  inner sep = 0.5mm, label=below:{3}] {};
	\node (3) at (2, 0) [circle, draw = black, fill=black, inner sep = 0.5mm, label=below:{1}] {};
    \node (4) at (0.5, 0.5) [circle, draw = black, fill=black, inner sep = 0.5mm, label=left:{\tiny \{2,3\}}] {};
    \node (5) at (1, 1) [circle, draw = black, fill=black, inner sep = 0.5mm, label=above:{\tiny \{1,2,3\}}] {};
    \node (6) at (1, -1.5) {$T_4$};
	\draw (1)--(4);
    \draw (2)--(4);
    \draw (4)--(5);
    \draw (3)--(5);
\end{tikzpicture}$$
\caption{Connected Gluing Assembly Trees for $K_3$}\label{K3CGR}
\end{figure}
\noindent
The connected gluing rule is far less restrictive than the edge gluing rule. The edge gluing rule produces only binary trees whereas the connected gluing rule allows for non-binary trees. 

In the next section, assembly trees with the addition of the connected gluing rule will be enumerated for the families of star, path, cycle, and complete graphs, illustrated below for 1-5 vertices:

$n=1,2$ for all:
$$
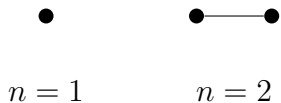
\begin{tikzpicture}
\def\Node{\node [circle, fill, inner sep=2pt]}
\Node at (0,1){};
\node at (0,0){$n=1$};

\Node (A) at (2,1){};
\Node (B) at (3,1){};
\draw (A)--(B);
\node at (2.5,0){$n=2$};
\end{tikzpicture}$$

$n=3,4,5$:
\newline\newline 
Star Graphs:
$$
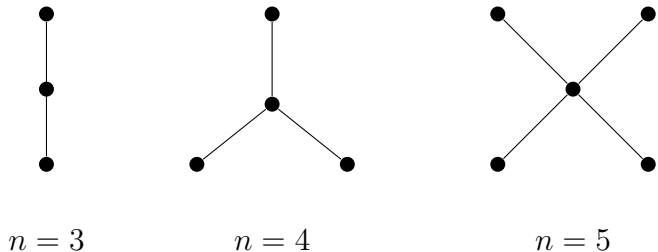
\begin{tikzpicture}
\def\Node{\node [circle, fill, inner sep=2pt]}
\Node (A) at (0,0){};
\Node (B) at (0,1){};
\Node (C) at (0,2){};
\draw (A)--(C);
\node at (0,-1){$n=3$};

\Node (D) at (3,0.8){};
\Node (E) at (3,2){};
\Node (F) at (2,0){};
\Node (G) at (4,0){};
\draw (E)--(D)--(F);
\draw (D)--(G);
\node at (3,-1){$n=4$};

\Node (H) at (7,1){};
\Node (I) at (6,2){};
\Node (J) at (6,0){};
\Node (K) at (8,2){};
\Node (L) at (8,0){};
\draw (I)--(H)--(K);
\draw (J)--(H)--(L);
\node at (7,-1){$n=5$};
\end{tikzpicture}$$
Path Graphs:
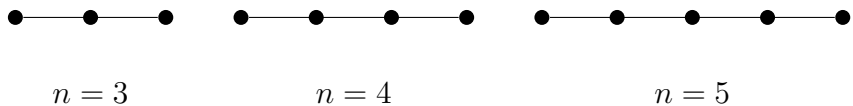
$$\begin{tikzpicture}
\def\Node{\node [circle, fill, inner sep=2pt]}
\node at (1,0){$n=3$};
\Node (A) at (0,1){};
\Node (B) at (1,1){};
\Node (C) at (2,1){};
\draw (A)--(C);

\node at (4.5,0){$n=4$};
\Node (D) at (3,1){};
\Node (E) at (4,1){};
\Node (F) at (5,1){};
\Node (G) at (6,1){};
\draw (D)--(G){};

\node at (9,0){$n=5$};
\Node (H) at (7,1){};
\Node (I) at (8,1){};
\Node (J) at (9,1){};
\Node (K) at (10,1){};
\Node (L) at (11,1){};
\draw (H)--(L);
\end{tikzpicture}$$
Cycle Graphs:
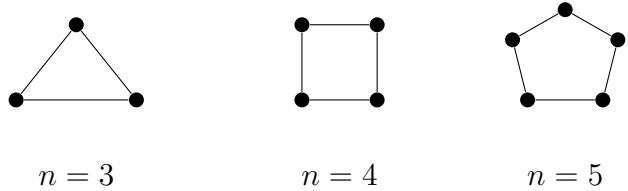
$$\begin{tikzpicture}
\def\Node{\node [circle, fill, inner sep=2pt]}
\node at (0,0){$n=3$};
\Node (A) at (0,2){};
\Node (B) at (-0.8,1){};
\Node (C) at (0.8,1){};
\draw (A)--(B)--(C)--(A);

\node at (3.5,0){$n=4$};
\Node (D) at (3,1){};
\Node (E) at (3,2){};
\Node (F) at (4,2){};
\Node (G) at (4,1){};
\draw (D)--(E)--(F)--(G)--(D);

\node at (6.5,0){$n=5$};
\Node (H) at (6,1){};
\Node (I) at (7,1){};
\Node (J) at (7.2,1.8){};
\Node (K) at (6.5,2.2){};
\Node (L) at (5.8,1.8){};
\draw (H)--(I)--(J)--(K)--(L)--(H);
\end{tikzpicture}$$
Complete Graphs:
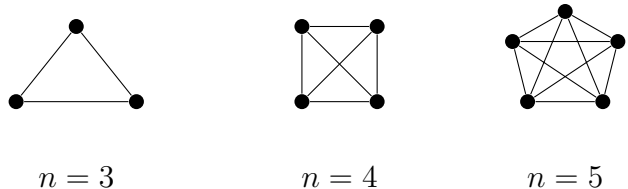
$$\begin{tikzpicture}
\def\Node{\node [circle, fill, inner sep=2pt]}
\node at (0,0){$n=3$};
\Node (A) at (0,2){};
\Node (B) at (-0.8,1){};
\Node (C) at (0.8,1){};
\draw (A)--(B)--(C)--(A);

\node at (3.5,0){$n=4$};
\Node (D) at (3,1){};
\Node (E) at (3,2){};
\Node (F) at (4,2){};
\Node (G) at (4,1){};
\draw (D)--(E)--(F)--(G)--(D)--(F);
\draw (E)--(G);

\node at (6.5,0){$n=5$};
\Node (H) at (6,1){};
\Node (I) at (7,1){};
\Node (J) at (7.2,1.8){};
\Node (K) at (6.5,2.2){};
\Node (L) at (5.8,1.8){};
\draw (H)--(I)--(J)--(K)--(L)--(H)--(J)--(L)--(I)--(K)--(H);
\end{tikzpicture}$$

\noindent
We denote the star graph on $n$ vertices by $S_n$, the path graph on $n$ vertices by $P_n$, the cycle graph on $n$ vertices by $C_n$, and the complete graph on $n$ vertices by $K_n$.

Throughout this paper, we assume that the vertices of $P_n$ are labeled by the elements of $[n]$ in increasing order from left to right. We also assume that the vertices of $C_n$ are labeled by the elements of $[n]$ in increasing clockwise order. Given the structure of $K_n$, there is no need to force any particular labeling on the vertices of these graphs.

\section{Enumeration of Assembly Trees}\label{enumer}

For a graph $G$, define $a^C(G)$ to be the number of assembly trees associated to $G$ satisfying the connected gluing rule.  We start our enumeration of assembly trees by first focusing on graphs with relatively simple connectivity structure. In particular we consider star, path, and cycle graphs.
\\*
\indent We begin with star graphs, $S_{n+1}$, where we disregard $S_1$. For our treatment of star graphs we label the middle vertex by 0, and the others by [n] as illustrated in Figure~\ref{fig:star}. Let $S_2(n,k)$ denote the Stirling numbers of the second kind.

\begin{figure}[H]
$$\begin{tikzpicture}[scale = 0.65]
	\node (0) at (0, 0) [circle, draw = black, inner sep = 0.5mm] {0};
	\node (1) at (0.62, 1.9) [circle, draw = black,  inner sep = 0.5mm] {1};
	\node (2) at (-1.62, 1.175) [circle, draw = black, inner sep = 0.5mm] {2};
	\node (3) at (-1.62, -1.175) [circle, draw = black, inner sep = 0.5mm] {3};
	\node (4) at (0.62, -1.9) [circle, draw = black,  inner sep = 0.5mm]{4};
	\node (5) at (2, 0) [circle, draw = black,  inner sep = 0.5mm]{5};
	\draw  (1)--(0);
	\draw  (2)--(0);
	\draw  (3)--(0);
	\draw  (4)--(0);
	\draw  (5)--(0);
\end{tikzpicture}$$
\caption{Labeled $S_6$}\label{fig:star}
\end{figure}
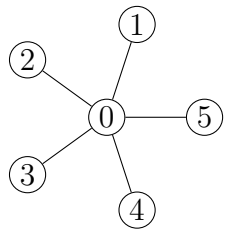

\begin{theorem}\label{egrspgf}
$a^C(S_{n+1})=\sum_{k=1}^{n}k!S_2(n, k)$ for $n\ge 1$.
\end{theorem}
\noindent
\begin{proof} For a subset of vertices of $S_{n+1}$ to correspond to a connected induced subgraph, the subset must contain the vertex labeled 0. Thus, given an appropriate labeling of the the non-leaf vertices of an assembly tree $U_1,\hdots, U_k$ we have $U_1\subset U_2\subset\hdots\subset U_k$. Therefore, $U_1-\{0\}, (U_2\backslash U_1)-\{0\},\hdots, (U_k\backslash\bigcup_{i=1}^{k-1}U_i)-\{0\}$ forms an ordered partition of $[n]$. Hence, assembly trees of $S_{n+1}$ are in bijection with ordered partitions of $[n]$ which are enumerated by $\sum_{k=1}^{n}k!S_2(n, k)$.
\end{proof}

\noindent
As stated in the introduction, this sequence of values is known as the sequence of Fubini numbers. Cayley \cite{C2} determined a generating function for this sequence of values given in the corollary below.

\begin{corollary}
Let $s_n=a^C(S_{n+1})$, then $\sum_{k=1}^{\infty}\frac{s_k}{k!}x^k=\frac{x}{2-e^x}$.
\end{corollary}

\noindent
Interestingly, when Cayley derived this generating function he was enumerating a family of trees which can be viewed as a generalization of assembly trees on $P_n$. We return to this realization at the end of this section.

Next, we focus on path graphs. For the following result let $SC(n)$ denote $n$th super Catalan number (A001003).

\begin{theorem}\label{thm:pc}
$a^C(P_n)=SC(n)$ for $n\ge 1$.
\end{theorem}
\noindent
\begin{proof} Note, we can always arrange the leaves of an assembly tree of $P_n$ to be in increasing order from left to right; this follows from the fact that two vertices in $P_n$ are connected by an edge if and only if they are labeled by consecutive integers in our fixed labeling.  Thus, there exists a bijection between assembly trees for $P_n$ and plane trees with $n$ leaves and with all internal vertices having two or more children. The latter is counted by the super Catalan numbers.
\end{proof}

\noindent
The generating function for the $SC(n)$, and thus the $a^C(P_n)$, is also well-known and given below as a corollary.

\begin{corollary}\label{egrspgf1}
Let $p_n=a^C(P_n)$, then $\sum_{k=1}^{\infty}p_kx^k=\frac{1+x-\sqrt{1-6x+x^2}}{4}$.
\end{corollary}

Now, since the cycle graph $C_n$ is just the path graph $P_n$ with an extra edge connecting its vertices of degree 1, we consider assembly trees for $C_n$ next.

\begin{theorem}\label{cycle}
The number of assembly trees for $C_n$ is $$a^C(C_n)=\sum_{k=2}^n\left[\sum_{i_1+...+i_k=n}i_1\prod_{j=1}^kSC(i_j)\right].$$
\end{theorem}
\begin{proof} Given the fixed labeling of the vertices of $C_n$ by $[n]$, connected components consist of either a consecutive block of $[n]$ or a consecutive block ending in $n$ and a consecutive block beginning with 1.

Let $U_1$ through $U_k$ for $2\le k\le n$ be the vertices adjacent to the root of an assembly tree of $C_n$. Now, assume that $U_1$ contains 1, $U_2$ contains the least element of $[n]\backslash U_1$,$\hdots$, and $U_k$ contains the least element of $[n]\backslash\bigcup_{i=1}^{k-1}U_i$. Note, that $|U_1|+\hdots+|U_k|$ forms a composition of $n$. Now, fix each $|U_i|$. Given the structure of $C_n$, the induced subgraphs of the $U_i$ will be paths. By Theorem~\ref{thm:pc}, the number of assembly trees for the induced subgraph of $U_i$ is $SC(|U_i|)$. Notice that with $|U_1|$ fixed, there are $|U_1|$ choices for the subset $U_1$. Furthermore, once $U_1$ has been fixed, all other $U_i$ get fixed as well. Thus, given a fixed $U_1$ there are $\prod_{i=1}^kSC(|U_i|)$ such assembly trees. Hence, there are $|U_1|\prod_{j=1}^kSC(|U_j|)$ assembly trees of $C_n$ with the given values for the $|U_i|$. Summing over all compositions with at least 2 parts the result follows.
\end{proof}

\begin{corollary}\label{egrscgf}
Let $c_n=a^C(C_n)$, then $\sum_{k=1}^{\infty}c_kx^k=\frac{x^2+x-x\sqrt{1-6x+x^2}}{4\sqrt{1-6x+x^2}}+x$.
\end{corollary}
\begin{proof}
Utilizing Theorem~\ref{cycle} and Corollary~\ref{egrspgf1} we get $$\sum_{k=2}^{\infty}c_kx^k=x\frac{d}{dx}\left[\frac{1+x-\sqrt{1-6x+x^2}}{4}\right]\left(\frac{1}{1-\frac{1+x-\sqrt{1-6x+x^2}}{4}}\right)-x\frac{d}{dx}\left[\frac{1+x-\sqrt{1-6x+x^2}}{4}\right]$$ $$=\frac{x^2+x-x\sqrt{1-6x+x^2}}{4\sqrt{1-6x+x^2}}.$$
\end{proof}

\noindent
The generating function in Corollary~\ref{egrscgf} is related to A047781 in the OEIS\footnote{A047781 is the sequence defined by $a(n)=\sum_{k=0}^{n-1}\binom{n-1}{k}\binom{n+k}{k}$. It is known that this sequence enumerates the number of lattice paths from $(0,0)$ to $(n,n)$ consisting of non-vertical segments.}; utilizing this relationship we obtain the following.

\begin{corollary}
$a^C(C_1)=1$, $a^C(C_n)=\sum_{i=0}^{n-2}\binom{n-2}{i}\binom{n+i-1}{i}=\sum_{k=0}^{n-2}\binom{n-2}{k}\binom{n-1}{k+1}2^k.$
\end{corollary}

All choices of families of graphs up to this point have had simple connectivity structure in the sense that the number of edges was relatively minimal. For the final family of graphs we consider complete graphs, which also have simple connectivity structure. In the complete case, the simplicity derives from the fact that all vertices are adjacent and so the number of edges is maximal. For complete graphs we obtain a recursive relation. Let $Part(n,k)$ denote the collection of integer partitions of $n$ into $k$ parts, i.e, $\lambda\in Part(n,k)$ implies $\lambda=(\lambda_1,\hdots,\lambda_k)$ with $\lambda_1\ge \lambda_2\ge\hdots\ge\lambda_k>0$ and $\lambda_1+\hdots+\lambda_k=n$.

\begin{theorem}
For $n\ge 2$ we have
$$a^C(K_n)=\sum_{k=2}^n\left[{\sum_{\lambda\in Part(n,k)}}\frac{\binom{n}{\lambda_1,...,\lambda_k}}{\prod_{i=1}^nm_i(\lambda)!}\prod_{j=1}^ka^C(K_{\lambda_j})\right]$$ where $m_i(\lambda)$ is the number of occurences of $i$ in the partition $\lambda$.
\end{theorem}
\begin{proof} Begin by considering the vertices adjacent to the root of a given assembly tree of $K_n$, say $V_1,\hdots,V_k$. Given the structure of $K_n$, any subset of the vertices produces a connected induced subgraph. In fact, for each $V_j$ the induced subgraph is $K_{|V_j|}$. Now, let's assume that $|V_1|\ge\hdots\ge|V_k|$ with $|V_1|+\hdots+|V_k|=n$, i.e., $\lambda=(|V_1|,\hdots,|V_k|)\in Part(n,k)$. Since the induced subgraph of $V_j$ must be $K_{|V_j|}$, we know that there are $a^C(K_{|V_j|})$ choices of assembly tree for the induced subgraph of $V_j$. Thus, there are $\prod_{j=1}^ka^C(K_{|V_j|})$ choices of assembly trees for the induced subgraphs of $V_1,\hdots, V_k$. Furthermore, there are $\frac{\binom{n}{|V_1|,...,|V_k|}}{\prod_{i=1}^nm_i(\lambda)!}$ choices of $V_1,\hdots,V_k$ with $|V_1|\ge\hdots\ge|V_k|$. Replacing $|V_j|$ by $\lambda_j$ the result follows.
\end{proof}

Returning to star graphs, in determining the generating function for the sequence of values $a^C(C_n)$ Cayley was enumerating a different variety of trees with certain restrictions on the degree of each node. This includes the enumeration of trees with $n$ leaves -- ``knots'' in Cayley's terminology -- recursively constructed from such trees on $<n$ leaves by adjoining branches emanating from the old leaves to the new leaves \cite{C2}. Examples for $n=1,2,3$ are illustrated below in Figure~\ref{Caytree}.

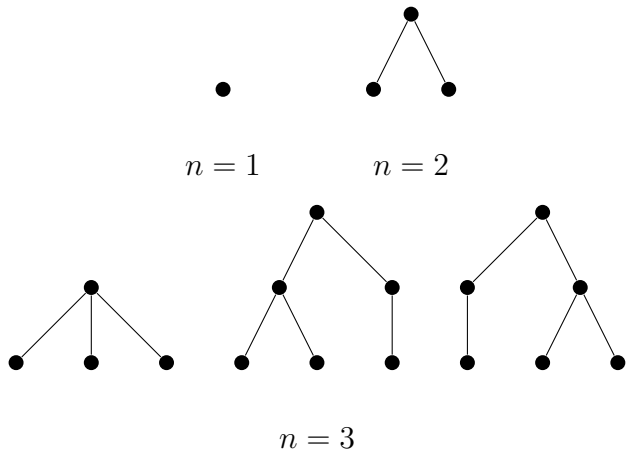
\begin{figure}[H]
$$\begin{tikzpicture}
\def\Node{\node [circle, fill, inner sep=2pt]}
\Node at (0,1){};
\node at (0,0){$n=1$};

\Node (A) at (2,1){};
\Node (B) at (3,1){};
\Node (C) at (2.5, 2){};
\draw (A)--(C)--(B);
\node at (2.5,0){$n=2$};
\end{tikzpicture}$$
$$\begin{tikzpicture}
\def\Node{\node [circle, fill, inner sep=2pt]}
\Node (1) at (0,1){};
\Node (2) at (1, 1){};
\Node (3) at (2,1){};
\Node (4) at (1, 2){};
\draw (1)--(4)--(3);
\draw (2)--(4);

\Node (5) at (3, 1){};
\Node (6) at (4, 1){};
\Node (7) at (5, 1){};
\Node (8) at (3.5, 2){};
\Node (9) at (5,2){};
\Node (10) at (4, 3){};
\draw (5)--(8)--(6);
\draw (7)--(9);
\draw (8)--(10)--(9);
\node at (4, 0) {$n=3$};

\Node (11) at (6, 1){};
\Node (12) at (7, 1){};
\Node (13) at (8, 1){};
\Node (14) at (6,2){};
\Node (15) at (7.5, 2){};
\Node (16) at (7, 3){};
\draw (11)--(14);
\draw (12)--(15)--(13);
\draw (15)--(16)--(14);
\end{tikzpicture}$$
\caption{Cayley's recursive trees}\label{Caytree}
\end{figure} 

\noindent
One can view this family of trees as a generalization of assembly trees on path graphs where there is a notion of time or order. In the next section we define such a generalization and enumerate the number of such assembly trees with both the edge and connected gluing rules for the same four families of graphs considered above.

\section{Time Dependent Assembly Trees}~\label{timdepprelim}

In this section a generalization of the assembly tree established in Section~\ref{prelim} is defined where the order or time at which vertices of $G$ are grouped is taken into account. Recall that for  a vertex $U$ in a rooted tree, $c(U)$ denotes the children of $U$.

\begin{definition} 
A \textit{time-dependent assembly tree} for a connected graph $G$ on $n$ vertices is a rooted tree, each node of which is labeled by a subset $U\subseteq V$ and a nonnegative integer $i$ such that
\begin{enumerate}
	\item there are leaves labeled $(v, 0),$ for each vertex $v\in V$,
	\item each internal (non-leaf) node has at least two children,
	\item the label on the root is $(V,m)$ for $1\le m\le n-1$,
	\item for each node $(U,i)$ with $i<m$, $U=\bigcup \{v\}$ for all $(v, 0)\in c((U, i))$,
    \item if $(U, i)$ and $(U',i')$ are adjacent nodes with $U\subseteq U'$, then $i<i'$,
    \item for each $0 \leq i \leq m$, there exists a node $(U,i)$ with $U\subseteq V$.
\end{enumerate}
\end{definition}

\noindent
An example of a time-dependent assembly tree is given below for an arbitrary graph on seven vertices.

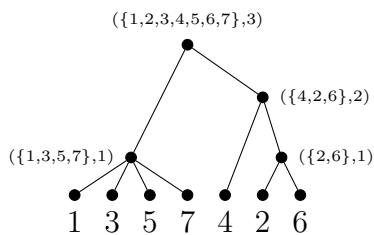
\begin{figure}[H]\label{assembly graph}
$$\begin{tikzpicture}
	\node (1) at (0, 3) [circle, draw = black, fill=black, inner sep = 0.5mm, label=above:{\tiny(\{1,2,3,4,5,6,7\},3)}] {};
	\node (2) at (-.75,1.5) [circle, draw = black, fill=black,  inner sep = 0.5mm, label=left:{\tiny(\{1,3,5,7\},1)}] {};
	\node (3) at (1, 2.3) [circle, draw = black, fill=black, inner sep = 0.5mm, label=right:{\tiny(\{4,2,6\},2)}] {};
    \node (4) at (1.25, 1.5) [circle, draw = black, fill=black, inner sep = 0.5mm, label=right:{\tiny(\{2,6\},1)}] {};
    \node (5) at (-1.5, 1) [circle, draw = black, fill=black, inner sep = 0.5mm, label=below:{1}] {};
    \node (6) at (-1, 1) [circle, draw = black, fill=black, inner sep = 0.5mm, label=below:{3}] {};
    \node (7) at (-.5, 1) [circle, draw = black, fill=black, inner sep = 0.5mm, label=below:{5}] {};
    \node (8) at (0, 1) [circle, draw = black, fill=black, inner sep = 0.5mm, label=below:{7}] {};
    \node (9) at (.5, 1) [circle, draw = black, fill=black, inner sep = 0.5mm, label=below:{4}] {};
    \node (10) at (1, 1) [circle, draw = black, fill=black, inner sep = 0.5mm, label=below:{2}] {};
    \node (11) at (1.5, 1) [circle, draw = black, fill=black, inner sep = 0.5mm, label=below:{6}] {};
	\draw (2)--(1)--(3);
    \draw (9)--(3)--(4);
    \draw (5)--(2)--(8);
    \draw (6)--(2)--(7);
    \draw (10)--(4)--(11);
\end{tikzpicture}$$
\caption{A time-dependent assembly tree}\label{tdassemtree}
\end{figure}

\noindent
To be clear, when we use the term assembly tree without the prefix time-dependent, we are referring to the object defined in Section~\ref{prelim}.

Note, that a time-dependent assembly tree of a graph $G$ is a mechanism which records not only the sequence in which the vertices are combined but also the time or stage at which they are combined in a given assembly of $G$. Thus, one assembly tree can correspond to multiple time-dependent assembly trees. This is illustrated in Figure~\ref{compare} where the assembly tree drawn above corresponds to all three time-dependent assembly trees drawn below it. To unclutter the assembly trees in Figure~\ref{compare} the labels have been removed from the internal nodes and it is assumed that the leaves are labeled in increasing order from left to right. Thus, the subsets corresponding to each node should be clear. Furthermore, the time integer labels for the time-dependent assembly trees are given by the horizontal dashed line on which the given node rests.

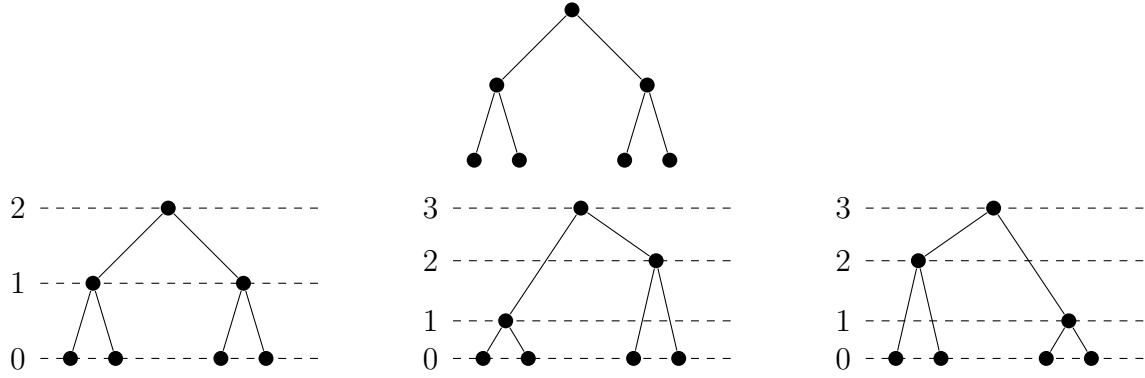
\begin{figure}[H]
$$\begin{tikzpicture}
\def\Node{\node [circle, fill, inner sep=2pt]}
\Node (A) at (0,3){};
\Node (B) at (-1,2){};
\Node (C) at (-1.3,1){};
\Node (D) at (-0.7,1){};
\Node (E) at (1,2){};
\Node (F) at (0.7,1){};
\Node (G) at (1.3,1){};
\draw (C)--(B)--(D);
\draw (F)--(E)--(G);
\draw (B)--(A)--(E);
\end{tikzpicture}$$
$$\begin{tikzpicture}
\def\Node{\node [circle, fill, inner sep=2pt]}
\Node (A) at (0,3){};
\Node (B) at (-1,2){};
\Node (C) at (-1.3,1){};
\Node (D) at (-0.7,1){};
\Node (E) at (1,2){};
\Node (F) at (0.7,1){};
\Node (G) at (1.3,1){};
\draw (C)--(B)--(D);
\draw (F)--(E)--(G);
\draw (B)--(A)--(E);
\draw [dashed](-1.7,1)--(2,1);
\draw [dashed](-1.7,2)--(2,2);
\draw [dashed](-1.7,3)--(2,3);
\node at (-2,1){0};
\node at (-2,2){1};
\node at (-2,3){2};
\end{tikzpicture}
\hspace{3em}
\begin{tikzpicture}
\def\Node{\node [circle, fill, inner sep=2pt]}
\Node (A) at (0,3){};
\Node (B) at (-1,1.5){};
\Node (C) at (-1.3,1){};
\Node (D) at (-0.7,1){};
\Node (E) at (1,2.3){};
\Node (F) at (1.3,1){};
\Node (G) at (0.7,1){};
\draw (C)--(B)--(D);
\draw (B)--(A)--(E);
\draw (G)--(E)--(F);
\draw [dashed](-1.7,1)--(2,1);
\draw [dashed](-1.7,1.5)--(2,1.5);
\draw [dashed](-1.7,2.3)--(2,2.3);
\draw [dashed](-1.7,3)--(2,3);
\node at (-2,1){0};
\node at (-2,1.5){1};
\node at (-2,2.3){2};
\node at (-2,3){3};
\end{tikzpicture}
\hspace{3em}
\begin{tikzpicture}
\def\Node{\node [circle, fill, inner sep=2pt]}
\Node (A) at (0,3){};
\Node (B) at (-1,2.3){};
\Node (C) at (1,1.5){};
\Node (D) at (0.7,1){};
\Node (E) at (1.3,1){};
\Node (F) at (-1.3,1){};
\Node (G) at (-0.7,1){};
\draw (G)--(B)--(F);
\draw (D)--(C)--(E);
\draw (B)--(A)--(C);
\draw [dashed](-1.7,1)--(2,1);
\draw [dashed](-1.7,1.5)--(2,1.5);
\draw [dashed](-1.7,2.3)--(2,2.3);
\draw [dashed](-1.7,3)--(2,3);
\node at (-2,1){0};
\node at (-2,1.5){1};
\node at (-2,2.3){2};
\node at (-2,3){3};
\end{tikzpicture}$$
\caption{Assembly Trees v. Time-Dependent Assembly Trees}\label{compare}
\end{figure}

Just as in the case of assembly trees without additional gluing rules, any two graphs with the same number of vertices have the exact same collection of time-dependent assembly trees. In the next few sections, time-dependent assembly trees with the addition of the edge and connected gluing rules are enumerated for the families of star, path, cycle, and complete graphs. We will adopt the following helpful notation: Given a time-dependent assembly tree T of a graph $G=(V, E)$ with $n$ vertices let
$$U_j(T)=\{U_i\subseteq V\, |\, \text{for }s\le j, (U_i,s)\text{ is assigned to a node of }T\}.$$

\noindent
Observe that $U_j(T)$ consists of all subsets of $V$ assigned to a vertex of $T$ corresponding to a nonnegative integer $i\le j$.  Now, we can define 
$$\P_j(T) = U_j(T) \backslash \{U_i\in U_j(T)\,|\,\exists U_s\in U_{j,T}\text{ such that }U_i\subseteq U_s\}.$$
Thus, $\P_j(T)$ contains all of the maximal subsets contained in $U_j(T)$ under inclusion. Note that $\P_j(T)$ forms a set partition of $[n]$.

As an example, consider the time-dependent assembly tree illustrated in Figure~\ref{tdassemtree}, call it $T$. For $T$, $$U_2(T)=\{\{1\},\{2\},\{3\},\{4\},\{5\},\{6\},\{7\},\{1,3,5,7\},\{2,6\},\{2,4,6\}\}$$

\noindent
so that

$$\P_2(T)=\{\{1,3,5,7\},\{2,4,6\}\}.$$

\subsection{Connected Gluing Rule}~\label{seccgr}

For a graph $G$, define $a_t^C(G)$ to be the number of time-dependent assembly trees associated to $G$ satisfying the connected gluing rule. 

As in Section~\ref{enumer}, star graphs are investigated first with $S_1$ disregarded. Given the structure of star graphs, it is clear that time-dependent assembly trees satisfying the connected gluing rule correspond exactly to assembly trees satisfying the connected gluing rule. We state the following proposition for completeness.

\begin{proposition}
$a_t^C(S_{n+1})=\sum_{k=1}^{n}k!S_2(n, k)$ for $n\ge 1$.
\end{proposition}

For completeness we also state the associated generating function result.

\begin{corollary}\label{egrspgf2}
If $s_n=a_t^C(S_{n+1})$, then $\sum_{k=1}^{\infty}\frac{s_k}{k!}x^k=\frac{x}{2-e^x}$.
\end{corollary}

As mentioned earlier, Cayley \cite{C2, C1} derived a generating function for the sequence $a^C_t(P_n)$, given in Corollary~\ref{egrspgf2}. Thus, we have the following non-obvious result:

\begin{proposition}
$a^C_t(P_n)=a^C_t(S_{n+1})=\sum_{k=1}^{n}k!S_2(n, k)$ for $n\ge 1$.
\end{proposition}

\noindent
\textit{Remark}: It is important to note that if we are given a time-dependent assembly tree, say $T$, of a graph $G$ with $|G|=n$, then by Part 1 of the definition of time-dependent assembly trees we must have $1\le |\P_1(T)|< n$.
\\*

Next we examine the family of cycle graphs.

\begin{proposition}\label{cgrcrr}
$a^C_t(C_n)=1+\sum_{j=2}^{n-1}\binom{n}{j}a^C_t(C_j)$ where $a^C_t(C_1)=a^C_t(C_2)=1$.
\end{proposition}
\begin{proof}
Note that connected components of $C_n$ correspond to non-crossing partitions of $[n]$ under our fixed labeling of the vertices of $C_n$ by $[n]$. Now, if all vertices of $C_n$ are not assembled at once, then $2\le |\P_1(T)|=j\le n-1$. Using our bijection with non-crossing partitions we find that the number of choices for such $\P_1(T)$ is $\binom{n}{j}$. Now, given the connectivity structure of the subsets in $\P_1(T)$, there are $a_t^C(C_j)$ ways to complete the assembly of $C_n$. Thus, the number of such time-dependent assembly trees $T$ of $C_n$ with $|\P_1(T)|=j$ is $\binom{n}{j}a_t^C(C_j)$. The result follows by noting that there is only one way to assemble all vertices of $C_n$ at once.
\end{proof}

The fact that the coefficients in the recursive relation of Proposition~\ref{cgrcrr} all have numerator $n!$ suggests that an exponential generating function for $a_t^C(C_n)$ is existent. Indeed, this is so:

\begin{corollary}\label{cgcgf}
Let $c_n=a_t^C(C_n)$, then $\sum_{k=1}^{\infty}\frac{c_k}{k!}x^k=\frac{x-xe^x+e^x-1}{2-e^x}$. 
\end{corollary}
\begin{proof}
Let $A(x)$ denote the desired generating function. Translating the recursive relation of Proposition~\ref{cgrcrr} into a relation satisfied by $A(x)$ gives the following: $$A(x)-x-\frac{x^2}{2}=(A(x)-x)(e^x-1)+e^x-x-\frac{x^2}{2}-1.$$ Solving for $A(x)$ completes the proof.
\end{proof}

Finally, we consider complete graphs. The next proposition provides a recursive relation for the $a_t^C(K_n)$.

\begin{proposition}
$a^C_t(K_n)=\sum_{j=1}^{n-1}S_2(n,j)a^C_t(K_j)$ where $a^C_t(K_1)=a^C_t(K_2)=1$.
\end{proposition}
\begin{proof}
Note that any subset of vertices of $K_n$ forms a connected component. Thus, the the collection of $\P_1(T)$ with $|\P_1(T)|=j$ for a time-dependent assembly tree $T$ for $K_n$ is in bijective correspondence with the collection of partitions of $[n]$ into $j$ parts which is enumerated by $S_2(n, j)$. Now, given the connectivity structure of the subsets in $\P_1(T)$, there are $a_t^C(K_j)$ ways to complete the assembly of $K_n$. Therefore, the number of time-dependent assembly trees $T$ of $K_n$ with $|\P_1(T)|=j$ is $S_2(n,j)a_t^C(K_j)$.
\end{proof}

\subsection{Edge Gluing Rule}~\label{secegr}

For a graph $G$, define $a_t^E(G)$ to be the number of time-dependent assembly trees associated to $G$ with the addition of the edge gluing rule.

We begin with star graphs. Due to a star graphs structure it is easy to see that the time-dependent assembly trees satisfying the edge gluing rule correspond exactly to assembly trees satisfying the edge gluing rule. The following proposition is essentially Proposition 1 of Vince and Bona \cite{BV13}.

\begin{proposition}\label{egrsf}
$a_t^E(S_n)=n!$.
\end{proposition}

The following corollary is immediate.

\begin{corollary}
Let $s_n=a_t^E(S_n)$, then $\sum_{k=1}^{\infty}\frac{s_k}{k!}x^k=\frac{1}{1-x}$.
\end{corollary}

Next, the family of path graphs are analyzed. 

\begin{proposition}\label{egpath}
$a^E_t(P_n)=\sum_{j=1}^{\lfloor n/2\rfloor}\binom{n-j}{n-2j}a^E_t(P_{n-j})$ where $a^E_t(P_1)=a^E_t(P_2)=1$.
\end{proposition}
\begin{proof}
Given our fixed labeling of the vertices of $P_n$, there exists a bijection between $\P_1(T)$ for time dependent assembly trees $T$ of $P_n$ with $|\P_1(T)|=n-j$ and compositions of $n$ consisting of $n-j$ 1's and $j$ 2's where $j>1$. The bijection is as follows: a 2 occurs as the $i$th part of the composition if and only if $(\{i,i+1\}, 1)$ is a node in the time-dependent assembly tree; otherwise, the $i$th part is a 1. Thus, the number choices for $\P_1(T)$ satisfying $|\P_1(T)|=n-j$ is $\binom{n-j}{n-2j}$. Any choice of $\P_1(T)$ for a time-dependent assembly tree $T$ of $P_n$ with $|\P_1(T)|=n-j$ can be completed to a full time-dependent assembly tree in $a_t^E(P_{n-j})$ ways.
\end{proof}

Considering the bijection used in the proof of the recursion in Proposition~\ref{egpath}, one is led to the following relation satisfied by the generating of the $a_t^E(P_n)$.

\begin{corollary}\label{egpcor}
Let $P(x)$ be the generating function for the $a_t^E(P_n)$, then $P(x)=\frac{x+P(x+x^2)}{2}$.
\end{corollary}
\begin{proof}
Replacing $x$ by $x+x^2$ in $P(x)$ produces a series of terms of the form $a_t^E(P_k)x^n$, where $n$ is the weight of a composition of length $k$ consisting of 1's and 2's. Collecting the coefficients of $x^n$ we get $\sum_{k=1}^nC_{1,2}(k,n)a_t^E(P_k)x^n$, where $C_{1,2}(k,n)$ is the number of compositions of $n$ into $k$ parts consisting of 1's and 2's; note that $C_{1,2}(n,n)=1$ and for $\lceil n/2 \rceil \leq k<n$ we have $C_{1,2}(k,n)=\binom{k}{2k-n}$. Thus, at $n=1$ we get the term $a_t^E(P_1)x$. For $k>1$, we get 
$$\sum_{k=1}^n C_{1,2}(k,n)a_t^E(P_k)x^n = \displaystyle\sum\limits_{k=\lceil n/2 \rceil}^{n-1}\binom{k}{2k-n}a_t^E(P_k)x^n+a_t^E(P_n)x^n \qquad \qquad $$
$$= \displaystyle\sum\limits_{j=1}^{\lfloor n/2 \rfloor} \binom{n-j}{n-2j} a_t^E(P_{n-j})x^n + a_t^E(P_n)x^n = a_t^E(P_n)x^n + a_t^E(P_n)x^n.$$

The result follows.
\end{proof}

Interestingly, in the OEIS the sequence of values formed by the $a_t^E(P_n)$ are listed (A171792) only as the coefficients of the generating function satisfying the relation given in Corollary~\ref{egpcor}. Thus, Corollary~\ref{egpcor} gives a previously unknown counting realization to the sequence A171792.

Now, just as before, given their similar structure to path graphs, cycle graphs are considered next.

\begin{proposition}\label{egrcg}
$a^E_t(C_n)=\sum_{j=1}^{\lfloor n/2\rfloor}\left(\binom{n-j}{n-2j}+\binom{n-j-1}{n-2j}\right)a^E_t(C_{n-j})$ where $a^E_t(C_1)=a^E_t(C_2)=1$.
\end{proposition}
\begin{proof}
Note that the only difference between cycles $C_n$ and paths $P_n$ is the edge $\{1,n\}$. If $(\{1,n\},1)$ is not a node, then as in the proof of Proposition~\ref{egpath} the number of such time-dependent assembly trees is  $\sum_{j=1}^{\lfloor n/2\rfloor}\binom{n-j}{n-2j}a^E_t(C_{n-j})$. On the other hand, if $(\{1,n\},1)$ is such a node, then applying similar reasoning we get $\binom{n-j-1}{n-2j}$ choices of $\P_1(T)$ with $|\P_1(T)|=j$ for a time-dependent assembly tree $T$ of $C_n$. Thus, there are $\sum_{j=1}^{\lfloor n/2\rfloor}\binom{n-j-1}{n-2j}a^E_t(C_{n-j})$ such time-dependent assembly trees. Putting these two cases together, the result follows.
\end{proof}

\noindent
Finally, we investigate complete graphs.

\begin{proposition}\label{egrccff}
$a^E_t(K_n)=\sum_{i=1}^{\lfloor\frac{n}{2}\rfloor}\frac{n!}{2^ii!(n-2i)!}a_t^E(K_{n-i})$ where $a_t^E(K_1)=a_t^E(K_2)=1$.
\end{proposition}
\begin{proof}
Given the structure of $K_n$, the collection of set partitions formed by the $\P_1(T)$ for time-dependent assembly trees $T$ of $K_n$ with $|\P_1(T)|=i$ are in bijection with set partitions of $[n]$ into blocks of size 1 or 2 with $i$ blocks of size 2. The number of such blocks is $$\frac{\prod_{j=0}^{i-1}\binom{n-2j}{2}}{i!}=\frac{n!}{2^ii!(n-2i)!}.$$ Given $|\P_1(T)|=i$, the number of ways to complete such a time-dependent assembly tree of $K_n$ is $a_t^E(K_{n-i})$.
\end{proof}

\section{Questions}~\label{seccon}

In this preliminary investigation we found generating functions for the sequences $a^C(S_n)$, $a^C(C_n)$, $a^C(P_n)$, $a_t^C(S_n)$, $a_t^C(P_n)$, $a_t^C(C_n)$, and $a_t^E(S_n)$. What about generating functions for $a^C(K_n)$, $a_t^C(K_n)$, $a_t^E(P_n)$, $a_t^E(C_n)$, and $a_t^E(K_n)$? All recurrence relations determined involve some form of binomial coefficient which suggest a product of exponential generating functions (see Corollary~\ref{cgcgf}). 

Further research might also include the enumeration of both regular and time-dependent assembly trees for families of graphs with more complex structure. Of particular interest [2, p.15] is the enumeration of both regular and time-dependent assembly trees for caterpillar graphs (see Figure~\ref{caterpillar}) for which no recurrence relation is known.
 
\begin{figure}[H]
$$\begin{tikzpicture}
\def\Node{\node [circle, fill, inner sep=2pt]}
\Node (1) at (0,0){};
\Node (2) at (1,0){};
\Node (3) at (2,0){};
\Node (4) at (3,0){};
\Node (5) at (0,1){};
\Node (6) at (1,1){};
\Node (7) at (2,1){};
\Node (8) at (3,1){};
\draw (1)--(5);
\draw (1)--(2);
\draw (2)--(6);
\draw (2)--(3);
\draw (3)--(7);
\draw (3)--(4);
\draw (4)--(8);
\end{tikzpicture}$$
\caption{A caterpillar}\label{caterpillar}
\end{figure}
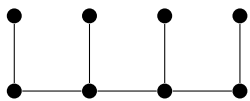

\bibliographystyle{abbrv}

\bibliography{bibliography_graph_assem.bib}

\end{document}